\documentclass[12pt]{article}
\usepackage[normalem]{ulem}

\usepackage{mathpazo}   
\usepackage{quoting}
\usepackage{tikz}
\usepackage{pgfplots}

\usetikzlibrary{arrows.meta}
\usetikzlibrary{backgrounds}
\usepgfplotslibrary{patchplots}
\usepgfplotslibrary{fillbetween}
\pgfplotsset{%
  layers/standard/.define layer set={%
    background,axis background,axis grid,axis ticks,axis lines,axis tick labels,pre main,main,axis descriptions,axis foreground%
  }{%
    grid style={/pgfplots/on layer=axis grid},%
    tick style={/pgfplots/on layer=axis ticks},%
    axis line style={/pgfplots/on layer=axis lines},%
    label style={/pgfplots/on layer=axis descriptions},%
    legend style={/pgfplots/on layer=axis descriptions},%
    title style={/pgfplots/on layer=axis descriptions},%
    colorbar style={/pgfplots/on layer=axis descriptions},%
    ticklabel style={/pgfplots/on layer=axis tick labels},%
    axis background@ style={/pgfplots/on layer=axis background},%
    3d box foreground style={/pgfplots/on layer=axis foreground},%
  },
}
\pgfplotsset{compat=1.18}  

\usepackage{amssymb}
\DeclareUnicodeCharacter{1D54F}{$\mathbb{X}$}

\usepackage{nameref}
\usepackage{empheq}
\usepackage{comment}
\usepackage[shortlabels,inline]{enumitem}
 \setlist[enumerate]{itemsep=0pt,parsep=0pt,topsep=0pt,partopsep=0pt}
\usepackage[colorlinks=true,
linkcolor=refkey,
urlcolor=lblue,
citecolor=red]{hyperref}
\usepackage[capitalize,nameinlink]{cleveref}
\usepackage{aliascnt} 
\usepackage[doc,wmm,hhb]{optional}
\usepackage{xcolor}

\usepackage{float}
\usepackage{soul}
\usepackage{graphicx}
\usepackage{subfig}

\usepackage{stmaryrd}
\usepackage{amssymb}
\usepackage[margin=1in]{geometry}

\definecolor{labelkey}{rgb}{0,0.08,0.45}
\definecolor{refkey}{rgb}{0,0.6,0.0}
\definecolor{Brown}{rgb}{0.45,0.0,0.05}
\definecolor{lime}{rgb}{0.00,0.8,0.0}
\definecolor{lblue}{rgb}{0.5,0.5,0.99}
\definecolor{OliveGreen}{rgb}{0,0.6,0}
\definecolor{tyrianpurple}{rgb}{0.4, 0.01, 0.24}

\colorlet{hlcyan}{cyan!30}

\makeatletter
\def\namedlabel#1#2{\begingroup
	\def\@currentlabel{#2}%
	\label{#1}\endgroup
}
\makeatother

\newcommand{\seppthree}{\setlength{\itemsep}{-3pt}}

\newcommand{\weakly}{\ensuremath{\:{\rightharpoonup}\:}}

\newcommand{\nnn}{\ensuremath{{n\in{\mathbb N}}}}
\newcommand{\knn}{\ensuremath{{k\in{\mathbb N}}}}
\newcommand{\thalb}{\ensuremath{\tfrac{1}{2}}}
\newcommand{\menge}[2]{\big\{{#1}~\big |~{#2}\big\}}

\newcommand{\fenv}[1]%
{\ensuremath{\,\overrightarrow{\operatorname{env}}_{#1}}}
\newcommand{\benv}[1]%
{\ensuremath{\,\overleftarrow{\operatorname{env}}_{#1}}}

\newcommand{\scal}[2]{\left\langle{#1},{#2}  \right\rangle}

\newcommand{\zeroun}{\ensuremath{\left]0,1\right[}}
\newcommand{\RR}{\ensuremath{\mathbb R}}
\newcommand{\RP}{\ensuremath{\mathbb{R}_+}}
\newcommand{\RPP}{\ensuremath{\mathbb{R}_{++}}}

\newcommand{\NN}{\ensuremath{\mathbb N}}

\DeclareMathOperator*{\argmin}{argmin}
\newcommand{\prox}{\ensuremath{\operatorname{Prox}}}

\newcommand{\cspan}{\ensuremath{\overline{\operatorname{span}}\,}}

\newcommand{\Fix}{\ensuremath{\operatorname{Fix}}}
\newcommand{\Id}{\ensuremath{\operatorname{Id}}}

\newcommand{\distsq}[2]{\operatorname{dist}^2\!\left(#1,#2\right)}

\newcommand{\minimize}[2]{\ensuremath{\underset{\substack{{#1}}}{\mathrm{minimize}}\;\;#2 }}

{\begin{list}{}{%
			\settowidth{\labelwidth}{\textrm{#1~}}%
			\setlength{\leftmargin}{\labelwidth+\labelsep}}}
{\end{list}}

\crefname{equation}{}{equations}
\crefname{chapter}{Appendix}{chapters}
\crefname{item}{}{items}
\crefname{enumi}{}{}
\crefname{appsec}{Appendix}{Appendices}

\newtheorem{theorem}{Theorem}[section]
\newaliascnt{lemma}{theorem}
\newtheorem{lemma}[lemma]{Lemma}
\aliascntresetthe{lemma}
\crefname{lemma}{Lemma}{Lemmas}
\Crefname{lemma}{Lemma}{Lemmas}

\newaliascnt{lem}{theorem}

\aliascntresetthe{lem}
\crefname{lem}{Lemma}{Lemmas}
\Crefname{lem}{Lemma}{Lemmas}

\newaliascnt{corollary}{theorem}

\aliascntresetthe{corollary}
\crefname{corollary}{Corollary}{Corollaries}
\Crefname{corollary}{Corollary}{Corollaries}

\newaliascnt{cor}{theorem}

\aliascntresetthe{cor}
\crefname{cor}{Corollary}{Corollaries}
\Crefname{cor}{Corollary}{Corollaries}

\newaliascnt{proposition}{theorem}
\newtheorem{proposition}[proposition]{Proposition}
\aliascntresetthe{proposition}
\crefname{proposition}{Proposition}{Propositions}
\Crefname{proposition}{Proposition}{Propositions}

\newaliascnt{prop}{theorem}

\aliascntresetthe{prop}
\crefname{prop}{Proposition}{Propositions}
\Crefname{prop}{Proposition}{Propositions}

\newaliascnt{definition}{theorem}
\newtheorem{definition}[definition]{Definition}
\aliascntresetthe{definition}
\crefname{definition}{Definition}{Definitions}
\Crefname{definition}{Definition}{Definitions}

\newaliascnt{defn}{theorem}

\aliascntresetthe{defn}
\crefname{defn}{Definition}{Definitions}
\Crefname{defn}{Definition}{Definitions}

\newaliascnt{thm}{theorem}

\aliascntresetthe{thm}
\crefname{thm}{Theorem}{Theorems}
\Crefname{thm}{Theorem}{Theorems}

\newaliascnt{example}{theorem}
\newtheorem{example}[example]{Example}
\aliascntresetthe{example}
\crefname{example}{Example}{Examples}
\Crefname{example}{Example}{Examples}

\newaliascnt{ex}{theorem}
\newtheorem{ex}[ex]{Example}
\aliascntresetthe{ex}
\crefname{ex}{Example}{Examples}
\Crefname{ex}{Example}{Examples}

\newaliascnt{fact}{theorem}
\newtheorem{fact}[fact]{Fact}
\aliascntresetthe{fact}
\crefname{fact}{Fact}{Facts}
\Crefname{fact}{Fact}{Facts}

\newaliascnt{remark}{theorem}

\aliascntresetthe{remark}
\crefname{remark}{Remark}{Remarks}
\Crefname{remark}{Remark}{Remarks}

\newaliascnt{rem}{theorem}

\aliascntresetthe{rem}
\crefname{rem}{Remark}{Remarks}
\Crefname{rem}{Remark}{Remarks}

\newcommand{\boxedeqn}[1]{%
	\begin{equation}\fbox{%
		\addtolength{\linewidth}{-2\fboxsep}%
		\addtolength{\linewidth}{-2\fboxrule}%
		\begin{minipage}{\linewidth}%
			\begin{equation}#1\\begin{equation}+5mm]\end{equation}%
		\end{minipage}%
	}\end{equation}%
}

\providecommand{\norm}[1]{\lVert#1\rVert}

\providecommand{\lip}{\beta}

\providecommand{\RR}{\mathbb{R}}

\newcommand{\fix}{\ensuremath{\operatorname{Fix}}}

\providecommand{\Id}{\operatorname{{ Id}}}

\providecommand{\NN}{\mathbb{N}}

\providecommand{\fix}{\operatorname{Fix}}

\providecommand{\Id}{\operatorname{Id}}

\providecommand{\spn}{\operatorname{span}}

\providecommand{\RR}{\mathbb{R}}
\providecommand{\NN}{\mathbb{N}}

\definecolor{myblue}{RGB}{12,35,68}

\definecolor{mybrown}{RGB}{120,75,50} 

\newcommand{\mybluebox}[1]{%
  \begingroup
  \setlength{\fboxsep}{6pt}
  \fcolorbox{myblue}{myblue!8}{$\displaystyle #1$}%
  \endgroup
}

\allowdisplaybreaks 
	
\begin{document}
										
										
										
										%

\author{ 
   Heinz H.\ Bauschke\thanks{Department of Mathematics, University of British Columbia, Kelowna, B.C.\ V1V~1V7, Canada. E-mail: \texttt{heinz.bauschke@ubc.ca}}
   \;\;and\;\;
   Walaa M.\ Moursi\thanks{Department of Combinatorics and Optimization, University of Waterloo, Waterloo, Ontario N2L~3G1, Canada. E-mail: \texttt{walaa.moursi@uwaterloo.ca}}
}

\title{\textsf{\textbf{Understanding FISTA's weak convergence:\\
A step-by-step introduction to the 
2025~milestone
}
}
}
\date{{\color{black}February 28, 2026}}
\maketitle
\begin{abstract}

Beck and Teboulle's FISTA for finding a minimizer of the sum of two convex functions is one of the most important 
algorithms of the past decades. While function 
value convergence of the iterates was known, the 
actual convergence of the iterates remained elusive until 
October 2025 when 
Jang and Ryu, as well as 
Bo\c{t}, Fadili, and Nguyen proved weak convergence. 

In this paper, we provide a gentle self-contained introduction to the proof of their remarkable result.
\end{abstract}
										{ 
                                            
											
										}

\section{Introduction}

Throughout, we assume that 
\begin{empheq}[box=\mybluebox]{equation}
\text{$X$ is a real Hilbert space,} 
\end{empheq}
with inner product $\scal{\cdot}{\cdot}\colon X\times X\to\RR$, and induced norm $\|\cdot\|$.
We also assume:
\begin{subequations}\label{intro}
\begin{empheq}[box=\mybluebox]{align}
& f : X \to \mathbb{R} \ \text{is convex and $\lip$-smooth, where } \lip \in \mathbb{R}_{++}, \label{30.1a} \\
& g : X \to \left]-\infty, +\infty\right] \ \text{is convex lower semicontinuous and proper}, \label{30.1b} \\
& F := f + g. \label{30.1c}
\end{empheq}
\end{subequations}
We study the problem
\begin{equation}
\label{e:theprob}
\minimize{x\in X}\ F(x)=f(x)+g(x).
\end{equation}
We set and assume that 
\begin{empheq}[box=\mybluebox]{equation}
\label{30.1d}
S := \argmin F\neq\varnothing\quad\text{and}\qquad \mu := \min F(X).
\end{empheq}
The \emph{Proximal Gradient Method (PGM)} to solve \cref{intro} iterates the operator $T$ 
\begin{empheq}[box=\mybluebox]{equation}
\label{30.1e}
T := \prox_{\frac{1}{\lip} g}\circ \big(\Id - \tfrac{1}{\lip}\nabla f\big)
\end{empheq}
for which $\Fix T = S$. 
When $g\equiv 0$, PGM reduces to the gradient descent method, 
while $f\equiv 0$ makes PGM coincide with the proximal point algorithm. 
The sequence $(T^kx_0)_\knn$ generated by the proximal
gradient operator is known to converge weakly\footnote{Strong convergence fails in general; see, e.g., Hundal's example \cite{Hundal}.} to some 
point in $S$, with rate $0\leq F(x_k)-\mu = \mathcal{O}\big(\tfrac{1}{k}\big)$. 

Arguably one of the most important algorithms 
(if not the most important) in convex optimization 
to solve  \cref{e:theprob} 
is Beck and Teboulle's FISTA algorithm, which we now describe:
Throughout the paper, we assume that 
$(t_k)_{k \in \mathbb{N}}$ is a parameter sequence of positive real numbers satisfying the following 
for every $k \in \mathbb{N}$:
\begin{subequations}
\label{30.2}
\begin{empheq}[box=\mybluebox]{align}
t_k &\ge \frac{k+2}{2} \ge 1 =: t_0,  
\label{eq:FS:i}
\\
t_k^{2} &\ge t_{k+1}^{2} - t_{k+1}. 
\label{eq:FS:ii}
\end{empheq}
\end{subequations}

Given any $x_0\in X$, FISTA generates two sequences 
$(x_k)_\knn,(y_k)_\knn$ in $X$ via $y_0 := x_0$ and 
the update formulas 
\begin{subequations}\label{30.3}
\begin{empheq}[box=\mybluebox]{align}
x_{k+1} &:= T y_k = \prox_{\frac{1}{\lip} g} \big( y_k - \tfrac{1}{\lip} \nabla f(y_k) \big), \label{30.3a} \\
y_{k+1} &:= x_{k+1} + \frac{t_k - 1}{t_{k+1}} \big( x_{k+1} - x_k \big). \label{30.3b}
\end{empheq}
\end{subequations}

In the smooth case, $g\equiv 0$, FISTA 
specializes to Nesterov's accelerated gradient descent \cite{Nest83}.

\begin{fact}
\label{thm:30.4}
\label{f:BT}
Recall our assumptions~\cref{intro}, \cref{30.1d} and \cref{30.1e}.
Then the FISTA sequence $(x_k)_\knn$ satisfies
$0 \leq F(x_k) - \mu \leq \frac{2\lip \distsq{x_0}{S}}{(k+1)^{2}}
= \mathcal{O}\big(\tfrac{1}{k^{2}}\big).$
In particular,
\begin{equation}
\label{eq:300}
F(x_k) \to \mu.
\end{equation}
\end{fact}
\begin{proof}
This is \cite[Theorem~4.4]{BT2009} 
(While this result is stated in a finite-dimensional space, 
the proof carries over without difficulty.)
See also \cite{Nest83} for the case 
when $g\equiv 0$. 
\end{proof}

\cref{f:BT} is extremely important because not only is 
the function value convergence rate improved from PGM's 
$0\leq F(x_k)-\mu 
= \mathcal{O}\big(\tfrac{1}{k}\big)$ to 
FISTA's $0\leq F(x_k)-\mu = \mathcal{O}\big(\tfrac{1}{k^2}\big)$ but it is also known that FISTA's rate is optimal (see \cite[Section~2.1.2]{Nesterov2nded}). 

Beck and Teboulle's original assignment of parameters 
occurs when the largest solution of the quadratic (in 
$t_{k+1}$) inequality \cref{eq:FS:ii} is chosen, i.e., when 
\begin{equation}
\label{BTpar}
t_{k+1} = \frac{1+\sqrt{4t_k^2+1}}{2}. 
\end{equation}
However, for their choice, the convergence of the sequence 
$(x_k)_\knn$ was unknown\footnote{For parameter sequences with guaranteed convergence, see, e.g., \cite{ChD15} and \cite{AttouchCabot2018}.} until October 2025 
when 
Jang and Ryu announced an AI-assisted proof 
(see \cite[Section~1.2]{JangRyu2025} and the X posts \cite{RyuX2025}). 
Independently and in parallel, Bo{\c t}, 
Fadili, and Nguyen proved weak convergence \cite{BFN2025}. Both teams considered different algorithms as well.
Their analyses and proof technique were inspired by 
continuous-time models; see, e.g., 
\cite{SuBoydCandes2016,AttouchCabot2018}.
For other recent related work, 
see \cite{MNPV} and \cite{Salzo}. 

\emph{
The goal of this paper is to provide an elementary self-contained exposition 
of the proof of weak convergence of the FISTA sequence $(x_k)_\knn$.} 
We do not claim priority to this remarkable result; instead,
our aim is 
to make this material easily accessible for teaching: 
indeed, 
undergraduate students who took introductory courses in 
convex optimization and in functional analysis\footnote{A course in functional analysis is optional if one is solely interested in the convergence in Euclidean space.} should be 
able to enjoy this proof.

The remainder of this paper is organized as follows. 
In \cref{s:sota}, {\color{black} we review two facts on FISTA
that were known in September 2025. 
}
First steps towards the convergence proof are taken 
in \cref{s:approch}. 
In \cref{s:BCCH}, we discuss a very useful tool for 
establishing convergence of real sequences. 
Salzo sequences, which capture the essential idea of 
Opial's Lemma, are considered in \cref{s:Salzo}. 
The final \cref{s:theproof} contains a structured proof
of weak convergence for FISTA, as well as an
example featuring a convex feasibility problem
with a solution set that is not a singleton. 

Our notation is standard and follows largely \cite{BC2017}.

\section{{\color{black} Two useful facts that were known in September 2025}} 

\label{s:sota}

The following results were essentially already contained 
in the original analysis \cite{BT2009} by Beck and Teboulle. 

\begin{fact}
\label{lem:29.2}
Recall \cref{intro} and \cref{30.1e}.
Let $(x,y)\in X\times X$, and set $y_{+} := T y$. Then
\begin{equation}\label{29.4}
F(x) - F(y_{+}) \ge \frac{\lip}{2} \|x - y_{+}\|^{2} - \frac{\lip}{2} \|x - y\|^{2}.
\end{equation}
\end{fact}
\begin{proof} See \cite[Theorem~10.16]{Beck}.
\end{proof}

The \emph{auxiliary sequence}
$(z_k)_\nnn$, defined by\footnote{We borrow this name from  \cite{BFN2025}.}  
\begin{empheq}[box=\mybluebox]{equation}
\label{auxseq} 
z_k := (1-t_k)x_k + t_k y_k,
\end{empheq}
satisfies 
$z_0 = y_0 = x_0$ and, for $\knn$, 
\begin{equation}
\label{e:xyz}
z_{k+1} = (1-t_k)x_k + t_kx_{k+1}
\end{equation}
and
$\big( 1-\tfrac{1}{t_{k+1}}\big) x_{k+1} 
+ \tfrac{1}{t_{k+1}}z_{k+1} = y_{k+1}$. 
Next, for $s\in S$ and $\knn$, we set 
\begin{empheq}[box=\mybluebox]{align}
\delta_k &:= F(x_k) - \mu = F(x_k) - F(s) \ge 0, \label{eq:delta-def}\\
\xi_{k+1}(s)&:=t_k^{2}\,\delta_{k+1}
+\tfrac{\lip}{2}\norm{z_{k+1}-s}^2. 
\label{eq:def:xi}
\end{empheq}

The following result is implicitly contained in 
many analyses of FISTA; including the original paper 
\cite{BT2009}. 
(See also 
\cite[Theorem~30.2]{BBM2023} 
and 
\cite[Proposition~1]{BFN2025}.)
We include the proof in \cref{app:zkbd} for the sake of completeness because the boundedness of $(z_k)_\knn$ 
wasn't utilized much until very recently.

\begin{fact}
\label{lem:key}
Let $x_0\in X$
 and let $(x_k)_{k\in\NN}$
 and $(y_k)_{k\in\NN}$ be the corresponding 
 FISTA sequences (see \cref{30.3}). 
 Let $s\in \fix T=S$, and recall the definitions 
 \cref{auxseq}, \cref{eq:delta-def}, and 
\cref{eq:def:xi}. 
Then
\begin{equation}
\label{e:ineq1}
0\le \cdots\le \xi_{k+1}\le \xi_{k}\le \cdots \le \xi_1\le \frac{\lip}{2}
\norm{x_0-s}^2.
 \end{equation}
Consequently, 
$(\xi_k(s))_\knn$ converges to a nonnegative real number, and 
the auxiliary sequence 
\begin{equation}
\label{e:zkbd}
\text{$(z_k)_\knn$ is bounded.}
\end{equation}
\end{fact}
\begin{proof}
See \cref{app:zkbd}.
\end{proof}

\section{Approaching the proof of weak convergence of FISTA} 

\label{s:approch}

In this section, we collect some properties that will make 
the eventual proof of weak convergence more structured and simpler. 

\subsection{The FISTA sequence $(x_k)_\knn$ is bounded, and its weak cluster points lie in $S$}

The following result was independently discovered very 
recently (see \cite{JangRyu2025} and \cite{BFN2025}).

\begin{theorem}
\label{thm:bddfista}
The FISTA sequence $(x_k)_{k\in\NN}$ (see \cref{30.3}) 
satisfies the following: 
 \begin{enumerate}
\item
\label{thm:bddfista:i}
$(x_k)_\knn$ is bounded. 
\item 
\label{thm:bddfista:ii}
Every weak cluster point of 
$(x_k)_{k\in\NN}$ lies in $S$.
 \end{enumerate}
\end{theorem}
\begin{proof}
\cref{thm:bddfista:i}: 
Solving \cref{e:xyz} for $x_{k+1}$ 
reveals 
\begin{equation}
x_{k+1} = \Big( 1-\frac{1}{t_k}\Big) x_k + \frac{1}{t_k}z_{k+1}
\end{equation}
as a convex combination of $x_k$ and $z_{k+1}$. 
Now $(z_k)_\knn$ is bounded (see \cref{e:zkbd}), 
say $\sigma := \sup_{\knn}\|z_k\| <+\infty$; thus, 
\cref{eq:FS:i} and the convexity of the norm yield
\begin{equation}
(\forall\knn)\quad 
\|x_{k+1}\| \leq \Big(1-\frac{1}{t_k} \Big)\|x_k\|+
\frac{1}{t_k}\|z_{k+1}\|
\leq\max\{ \|x_k\|,\sigma\}.
\end{equation}
Inductively, we see that $(\forall\knn)$
$\|x_k\| \leq \max\{ \|x_0\|,\sigma\}<+\infty$. 

\cref{thm:bddfista:ii}:
Let $\overline{x}$
be a weak cluster point of
$(x_k)_{k\in\NN}$, say $x_{n_k}\weakly \overline{x}$.
The (weak) lower semicontinuity of $F$
and \cref{eq:300}
imply 
\begin{equation}
   \mu\le F(\overline{x})\le 
   \varliminf_{k\to\infty}\ F(x_{n_k}) =
   \lim_{k\to\infty} F(x_{n_k})=
   \lim_{k\to\infty} F(x_{k})= \mu.
\end{equation}
Therefore, $\mu = F(\overline{x})$ and hence 
$\overline{x}\in S$. 
\end{proof}

\subsection{A note on the parameter sequence $(t_k)_\knn$}

\begin{lemma}
\label{l:Ryu} 
The FISTA parameter sequence $(t_k)_{\knn}$ 
(see \cref{30.2}) satisfies $1 \leq t_k-1\leq k$ for $k\geq 2$. 
Consequently, 
\begin{equation}
\sum_{k=2}^\infty \frac{1}{t_k-1} \geq 
\sum_{k=2}^\infty \frac{1}{k} = +\infty. 
\end{equation}
\end{lemma}
\begin{proof}
Let $k\geq 2$. 
By \cref{eq:FS:i}, $t_k\geq (k+2)/2\geq (2+2)/2=2$ and so 
$t_k-1\geq 1$. 
On the other hand, 
by \cref{eq:FS:ii}, 
$t_{k+1} \leq \tfrac{1+\sqrt{1+4t_k^2}}{2} \leq 
\tfrac{1+(1+2t_k)}{2} = 1+t_k$.
Consequently, $t_k\leq k+t_0=k+1$. 
Altogether, $1\leq t_k-1\leq k$ and the result follows
(compare to the harmonic series). 
\end{proof}

\section{The Bo\c{t}-Chenchene-Csetnek-Hulett (BCCH) Lemma} 

\label{s:BCCH}

\subsection{The BCCH Lemma: statement and proof}

The following result is due to 
Bo\c{t}, Chenchene, Csetnek, and Hulett 
(see \cite[Lemma~A.4]{BCCH}). 
We present a slight extension (allowing the limit to be 
$\pm\infty$) with a slightly different, more elementary and detailed, proof. 

\begin{lemma}[the BCCH Lemma] 
\label{l:Radu}
Let $(\varphi_k)_\knn$ be a sequence of positive numbers, 
and let $(h_k)_\knn$ be a sequence of real numbers such that 
the following hold:
\begin{equation}
\varphi_k>0,\quad 
\sum_{k\in\NN} \frac{1}{\varphi_k} = +\infty, \quad
\text{and}\quad 
g_k := h_{k+1} + \varphi_k(h_{k+1}-h_k)\to \ell\in
[-\infty,+\infty]. 
\end{equation}
Then 
\begin{equation}
h_k \to \ell.
\end{equation}
\end{lemma}
\begin{proof}
{\color{black}\textbf{Step~1: Collecting properties 
related solely to $(\varphi_k)_\knn$}
}

On the one hand, 
the divergent-series assumption implies
\begin{equation}
\sum_\knn \min\big\{1,\tfrac{1}{\varphi_k}\big\} 
= +\infty. 
\end{equation}
On the other hand, 
$1+\varphi_k \leq 2\max\{1,\varphi_k\}$ and so 
\begin{equation}
\frac{1}{1+\varphi_k} \geq 
\frac{1}{2\max\{1,\varphi_k\}} = 
\tfrac{1}{2} \min\big\{1,\tfrac{1}{\varphi_k}\big\}. 
\end{equation}
Altogether, 
\begin{equation}
\label{261012a}
\sum_{\knn}
\frac{1}{1+\varphi_k} = +\infty. 
\end{equation}
Now set 
\begin{equation}
\lambda_k := \frac{\varphi_k}{1+\varphi_k} \in \zeroun;
\;\;\text{thus, }
1-\lambda_k = \frac{1}{1+\varphi_k}\in\zeroun.
\end{equation}
Hence \cref{261012a} yields 
\begin{equation}
\label{261012b}
\sum_{\knn}
(1-\lambda_k) = +\infty. 
\end{equation}
{\color{black}
Because $\ln(1-x)\leq -x$ for $x\in \left[0,1\right[$, 
we deduce from \cref{261012b} that 
}
for every $m\in\NN$, we have 
\begin{equation}
\ln\Big(\prod_{j=m}^n \lambda_j\Big)
= \sum_{j=m}^n \ln \big(1-(1-\lambda_j)\big)
\leq -\sum_{j=m}^n (1-\lambda_j) 
\to -\infty \quad \text{as $n\to\infty$.}
\end{equation}
In turn, this yields 
\begin{equation}
\label{260112c}
(\forall m\in\NN) \quad 
\lim_{n\to\infty} \prod_{j=m}^n \lambda_j = 0.
\end{equation}
Now define 
\begin{equation}
\label{260118a}
(\forall n\in\NN)(\forall k\in\NN)\quad 
w_{n,k} := 
(1-\lambda_k)\prod_{j=k+1}^{n-1} \lambda_j > 0, 
\end{equation}
and note that \cref{260112c} yields 
\begin{equation}
\label{260112h}
(\forall\knn) 
\quad 
\lim_{n\to\infty} w_{n,k} = 0. 
\end{equation}
Telescoping 
$w_{n,k} = \prod_{j=k+1}^{n-1}\lambda_j 
- \prod_{j=k}^{n-1}\lambda_j$
from $k=n-1$ down to  $k=0$ gives
\begin{align}
\label{260112i}
\sum_{k=0}^{n-1} w_{n,k} 
&= \prod_{j=n}^{n-1}\lambda_j 
- \prod_{j=0}^{n-1}\lambda_j
= 1 - \prod_{j=0}^{n-1}\lambda_j,
\end{align}
which implies, by using again \cref{260112c}, 
\begin{equation}
\label{260112f}
\lim_{n\to\infty} \sum_{k=0}^{n-1}w_{n,k} = 1.
\end{equation}
{\color{black}
\textbf{Step~2: Involving the sequences $(g_k)_\knn$  
and  $(h_k)_\knn$} 
}

First, \cref{260118a} yields 
\begin{equation}
\label{260112d}
\sum_{k=0}^{n} w_{n+1,k} g_k 
= (1-\lambda_n)g_n + 
\lambda_n \sum_{k=0}^{n-1} w_{n,k} g_k. 
\end{equation}
Next, the definition of $g_k$ yields 
$g_k + \varphi_kh_k = (1+\varphi_k)h_{k+1}$; hence, 
\begin{equation}
\label{260118b}
h_{k+1} = (1-\lambda_k)g_k + \lambda_k h_k. 
\end{equation}
Using \cref{260118b}, \cref{260112d}, and an induction on $n$, we obtain  
\begin{equation}
\label{260112e}
h_n = \sum_{k=0}^{n-1}w_{n,k}g_k + h_0\prod_{j=0}^{n-1}\lambda_j. 
\end{equation}

\textbf{Case~1: $\ell\in\RR$}\\
Using \cref{260112e}, we write
\begin{equation}
\label{260112g}
h_n -\ell = 
h_0\prod_{j=0}^{n-1}\lambda_j
+\ell\Big( -1+\sum_{k=0}^{n-1}w_{n,k}\Big)
+ \sum_{k=0}^{n-1}w_{n,k}(g_k-\ell). 
\end{equation}
Recalling \cref{260112c} and \cref{260112f}, we note that 
\begin{equation}
\label{260112j}
\lim_{n\to\infty} h_0\prod_{j=0}^{n-1}\lambda_j = 0 
\quad\text{and}\quad
\lim_{n\to\infty} \ell\Big( -1+\sum_{k=0}^{n-1}w_{n,k}\Big) = 0. 
\end{equation}
We now turn to the right-most sum in \cref{260112g}. 
Let $\varepsilon > 0$. 
Grab $K\in\NN$ such that 
$(\forall k\geq K)$
$|g_k-\ell|\leq\varepsilon$. 
For $n\geq K$, we estimate
\begin{equation}
\Big|\sum_{k=0}^{n-1}w_{n,k}(g_k-\ell) \Big|
\leq 
\sum_{k=0}^{K-1} w_{n,k}|g_k-\ell| 
+ \varepsilon \sum_{k=K}^{n-1}w_{n,k}.
\end{equation}
When $0\leq k \leq K-1$, we have 
$\lim_{n\to\infty} w_{n,k}=0$ by 
\cref{260112h}; thus, 
$\lim_{n\to\infty}\sum_{k=0}^{K-1} w_{n,k}|g_k-\ell| = 0$. 
On the other hand, 
$0\leq \varepsilon \sum_{k=K}^{n-1}w_{n,k}\leq \varepsilon$ by 
\cref{260112i}. 
Altogether, 
$\lim_{n\to\infty}|\sum_{k=0}^{n-1}w_{n,k}(g_k-\ell) |\leq \varepsilon$. 
Because $\varepsilon$ was chosen arbitrarily, we deduce that 
\begin{equation}
\label{260112k}
\lim_{n\to\infty} \sum_{k=0}^{n-1}w_{n,k}(g_k-\ell) = 0.
\end{equation}
Combining 
\cref{260112j}, \cref{260112k}, and \cref{260112g}, we deduce that $\lim_{n\to\infty}(h_n-\ell)=0$, and we are done. 

\textbf{Case~2: $\ell=\pm\infty$}\\
While cute, this case is not needed in the sequel so 
we put the proof in \cref{app:ellpm}. 
\end{proof}

\subsection{Optional Bonus material on the BCCH Lemma}

Readers intrigued by \cref{l:Radu} will enjoy the following
collection of limiting examples in this optional subsection.

\begin{example}[converse of the BCCH Lemma fails with $(g_k)_\knn$ bounded]
For $k\geq 1$, set $\varphi_k := k>0$, 
let $\ell\in\RR$, and set 
\begin{equation}
h_k := \ell + \frac{(-1)^k}{k} 
\quad\text{and}\quad
g_k := h_{k+1} + \varphi_k\big(h_{k+1}-h_{k}\big). 
\end{equation}
Then $h_k\to \ell$, yet 
$g_k =  \ell+(-1)^{k+1}2 \not\to\ell$. 
\end{example}

\begin{example}[converse of the BCCH Lemma fails with $(g_k)_\knn$ unbounded]
For $k\geq 1$, set $\varphi_k := k>0$, 
let $\ell\in\RR$, and set 
\begin{equation}
h_k := \ell + \frac{(-1)^k}{\sqrt{k}} 
\quad\text{and}\quad
g_k := h_{k+1} + \varphi_k\big(h_{k+1}-h_{k}\big). 
\end{equation}
Then $h_k\to \ell$, yet 
$g_k = \ell+(-1)^{k+1}
\big(\sqrt{k+1}+\sqrt{k}\big)
\not\to\ell$. 
\end{example}

\begin{example}[importance of the divergent-series 
condition in the BCCH Lemma]
\ \\
\label{ex:sinh}
For $k\geq 1$, set 
$\varphi_k := k^2$ so that 
$\sum_{k=1}^{\infty} 1/\varphi_k < +\infty$. 
Now set $h_1 := 1$ and, for $k\geq 1$, 
\begin{equation}
h_{k+1} := \frac{\varphi_k}{1+\varphi_k}h_k 
\quad\text{and}\quad
g_k 
:= 
h_{k+1}+\varphi_k(h_{k+1}-h_k)
\end{equation}
Then 
$g_k\equiv \ell=0$ but 
$h_k \to \pi/\sinh(\pi)\approx 0.272 > 0 = \ell$.
\end{example}
\begin{proof}
See \cref{app:sinh} for details.
\end{proof}

\section{Beyond Fej\'er and Opial: Salzo sequences}

\label{s:Salzo}

\subsection{Bounded Salzo sequences are weakly convergent}

Soon after the submission of \cite{JangRyu2025} and 
\cite{BFN2025}, 
Saverio Salzo presented an extension to allow for 
inexact computations and stochastic gradients. 
His proof makes use 
of a generalization of Opial's Lemma
(see \cite[Remark~2.5]{Salzo}). 
This motivates the following: 

\begin{definition}[Salzo sequence] 
We say that a sequence $(x_k)_\knn$ in $X$ 
is a \emph{Salzo sequence} if 
\begin{equation}
\label{260113a}
(\scal{x_k}{w_1-w_2})_\nnn 
\;\;\text{converges,}
\end{equation}
whenever 
$w_1,w_2$ are weak cluster points of $(x_k)_\knn$
\end{definition}

\begin{proposition}
\label{p:Salzo}
Let $(x_k)_\knn$ be a bounded Salzo sequence. 
Then $(x_k)_\knn$ is weakly convergent. 
\end{proposition}
\begin{proof}
Let $w_1,w_2$ be two (possibly different) 
weak cluster points of $(x_\knn)$, 
say $x_{m_k}\weakly w_1$ and $x_{n_k} \weakly w_2$. 
By assumption, 
$\ell= \lim_{k\to\infty} \scal{x_k}{w_1-w_2}$ exists. 
Going along the two subsequences yields 
$\scal{w_1}{w_1-w_2}=
\lim_{k\to\infty} \scal{x_{m_k}}{w_1-w_2} =
\ell=
\lim_{k\to\infty} \scal{x_{n_k}}{w_1-w_2} =
\scal{w_2}{w_1-w_2}$.
It follows that $0=\ell-\ell = \scal{w_1-w_2}{w_1-w_2}=
\|w_1-w_2\|^2$, and we're done.
\end{proof}

\subsection{Optional Bonus material on Salzo sequences}

\begin{proposition}
Let $(x_k)_\knn$ be a sequence that is 
Fej\'er monotone with respect to some nonempty subset 
$C$ of $X$, i.e., 
$(\forall c\in C)$ $(\forall \knn)$
$\|x_{k+1}-c\|\leq\|x_k-c\|$.
If all weak cluster points of $(x_k)_\knn$ lie in $C$, then
$(x_k)_\knn$ is a bounded Salzo sequence 
(and hence weakly convergent to some point in $C$). 
\end{proposition}
\begin{proof}
If $w_1,w_2$ are weak cluster points of 
$(x_k)_\knn$, then 
$(\|x_k-w_1\|^2-\|x_k-w_2\|^2)_\knn$ is convergent 
and so \cref{p:Salzo} applies.
\end{proof}

The following result features an assumption that 
generalizes the Salzo property.

\begin{proposition}
\label{p:bonusSalzo}
Let $(x_k)_\knn$ be a sequence in $X$, and 
let $C$ be a nonempty subset of $X$ such that 
\begin{equation}
(\scal{x_k}{c})_\knn 
\;\;\text{converges for all $c\in C$.}
\end{equation}
Set $Y := \cspan C$. 
Then $(P_Yx_k)_\knn$ converges weakly to some point 
$\overline{y}\in Y$. 
Moreover, $(\forall y\in Y)$
$\lim_{k\to\infty} \scal{x_k}{y} = \scal{\overline{y}}{y}$. 
\end{proposition}
\begin{proof}
See \cref{app:Salzo} for details.
\end{proof}

\section{The grand finale: FISTA converges weakly!}

\label{s:theproof}

\begin{theorem}[Bo\c{t}-Fadili-Nguyen and Jang-Ryu, 2025] 
{\rm \cite{BFN2025,JangRyu2025}}
\label{t:main}
Let $(x_k)_\knn$ and $(y_k)_\knn$ be the FISTA sequences 
(see \cref{30.3}). 
Then there exists $s\in S$ such that 
\begin{equation}
x_k\weakly s, \quad
x_k-y_k\to 0, \quad\text{and}\quad
y_k\weakly s. 
\end{equation}
\end{theorem}
\begin{proof}
We break up the proof into several steps.
We will first show that $(x_k)_\knn$ converges weakly
to some point in $S$.
To this end, we assume that 
\begin{equation}
\label{w1w2}
\text{
$w_1,w_2$ are two (possibly different) 
weak cluster points of $(x_k)_\knn$.
}
\end{equation}

\textbf{Claim~1:} 
$(\scal{z_k}{w_1-w_2})_\knn$ converges. \\
By \cref{thm:bddfista}\cref{thm:bddfista:ii}, 
$\{w_1,w_2\}\subseteq S$. 
Hence, by \cref{lem:key}, 
$(\xi_k(w_1))_\knn$ and 
$(\xi_k(w_2))_\knn$ are convergent sequences. 
Recalling the definition of $\xi_k$ (see \cref{eq:def:xi}), 
we deduce that 
the difference sequence 
$(\xi_k(w_1)-\xi_k(w_2))_\knn
=(\tfrac{\lip}{2}(\|z_k-w_1\|^2 - \|z_k-w_2\|^2))_\knn$
is convergent as well. 
Expanding and simplifying the latter sequence gives convergence of 
$(\tfrac{\lip}{2}(-2\scal{z_k}{w_1-w_2}+\|w_1\|^2-\|w_2\|^2))_\knn$. This completes the proof of 
\textbf{Claim~1}. 

Now set 
\begin{equation}
h_k := \scal{x_k}{w_1-w_2}
\quad\text{and}\quad
g_k := h_{k+1}+(t_k-1)(h_{k+1} - h_k).
\end{equation}

\textbf{Claim~2:} 
$(\forall\knn)$ 
$g_k = \scal{z_{k+1}}{w_1-w_2}$.\\
Indeed, 
\begin{align*}
g_k 
&= h_{k+1} + (t_k-1)(h_{k+1}-h_k) \tag{by definition of $g_k$}\\
&= 
(1-t_k)\scal{x_k}{w_1-w_2}+t_k\scal{x_{k+1}}{w_1-w_2}
\tag{by definition of $h_k$}\\
&= 
\scal{(1-t_k)x_k+t_kx_{k+1}}{w_1-w_2}\\
&= 
\scal{z_{k+1}}{w_1-w_2}, \tag{by \cref{e:xyz}}
\end{align*}
which verifies \textbf{Claim~2}. 

\textbf{Claim~3:} 
$(\scal{x_k}{w_1-w_2})_\knn$ converges. \\
By \textbf{Claim~2}, $(g_k)_\knn$ is convergent. 
Set $(\varphi_k)_{k\geq 2} := (t_k-1)_{k\geq 2}$. 
\cref{l:Ryu} implies that 
$(\varphi_k)_{k\geq 2}$ lies in $\RPP$ and 
$\sum_{k\geq 2} \tfrac{1}{\varphi_k}=+\infty$. 
Altogether, the BCCH \cref{l:Radu} yields 
the convergence of $(h_k)_\knn$, which is precisely
what \textbf{Claim~3} states. 

\textbf{Claim~4}: $(x_k)_\knn$ is a bounded Salzo sequence.\\
The boundedness follows from
\cref{thm:bddfista}\cref{thm:bddfista:i}.
And \cref{w1w2} and \textbf{Claim~3} imply that 
$(x_k)_\knn$ is a Salzo sequence. 

\textbf{Claim~5}:
$(x_k)_\knn$ converges weakly to some point $s\in S$.\\
Weak convergence follows from \textbf{Claim~4} and \cref{p:Salzo}. 
The weak limit $s$ must lie in $S$
by \cref{thm:bddfista}\cref{thm:bddfista:ii}.

\textbf{Claim~6}:
$y_k-x_k\to 0$.\\
Recall that $(x_k)_\knn$ and $(z_k)_\knn$ 
are bounded (see \textbf{Claim~4} and \cref{e:zkbd}) 
while $t_k \to +\infty$ (by \cref{eq:FS:i}). 
On the other hand, $y_k-x_k = \tfrac{1}{t_k}(z_k-x_k)$ by 
\cref{auxseq}. 
Altogether, \textbf{Claim~6} follows.

\textbf{Claim~7}:
$(y_k)_\knn$ also converges weakly to $s$.\\
Combine \textbf{Claim~5} with \textbf{Claim 6}. 

This completes the proof of the theorem. 
\end{proof}

We conclude the paper with an illustration of 
\cref{t:main}. 

\begin{example}[a feasibility problem]
Suppose that $X=\RR^2$, 
$f = \thalb\distsq{\cdot}{U}$, 
and 
$g = \iota_V$, 
where $U:=\RP^2$ and 
$V := \menge{(v_1,v_2)\in\RR^2}{v_1+v_2=1}$. 
Then $f$ is $\lip$-smooth, with $\beta=1$, 
and the line segment 
$S = U \cap V = [(0,1),(1,0)]$ is not a singleton. 
Thanks to \cref{t:main}, every FISTA sequence 
$(x_k)_\knn$ is now known to converge 
to some point in $S$. 
See \cref{fig:myplot} for a visualization of the first few iterates of FISTA with starting point 
$x_0 = (5,0)$ and parameter sequence \cref{BTpar}.
Using {\rm \texttt Julia}, we find numerically that 
$\lim_{k\to\infty} x_k \approx (0.4829,0.5171)$ 
which lies in the interior of $U$. 
Interestingly, our numerical result coincides with the answers provided\footnote{
Grok~4 and ChatGPT-5.2 were prompted on 
January 19, 2026 with: \texttt{
Run FISTA in the Euclidean plane, with the usual 
Back-Teboulle parameters. Starting point is 
\texttt{\$(5,0)\$}. The differentiable function (Lipschitz-1) is \texttt{\$1/2\$} times the distance squared, where we measure the distance to the nonnegative orthant. The gradient step is thus a simple projection onto the orthant. The other function is the indicator function of the line \texttt{\$x+y=1\$}. This also has a closed form. Compute the first 5 iterates of FISTA. Can you come up with an analytical formula? If not, can you approximate the limit of the FISTA iteration numerically?
} Neither we nor the two LLMs were able to find the limit analytically. 
\texttt{Gemini~3} and \texttt{Claude~Sonnet~4.5} suggested that the limit is $(1,0)$ and $(\thalb,\thalb)$, respectively. 
} by 
{\rm\texttt Grok} and by {\rm\texttt ChatGPT}. 

\textup{\texttt{\noindent 
}}

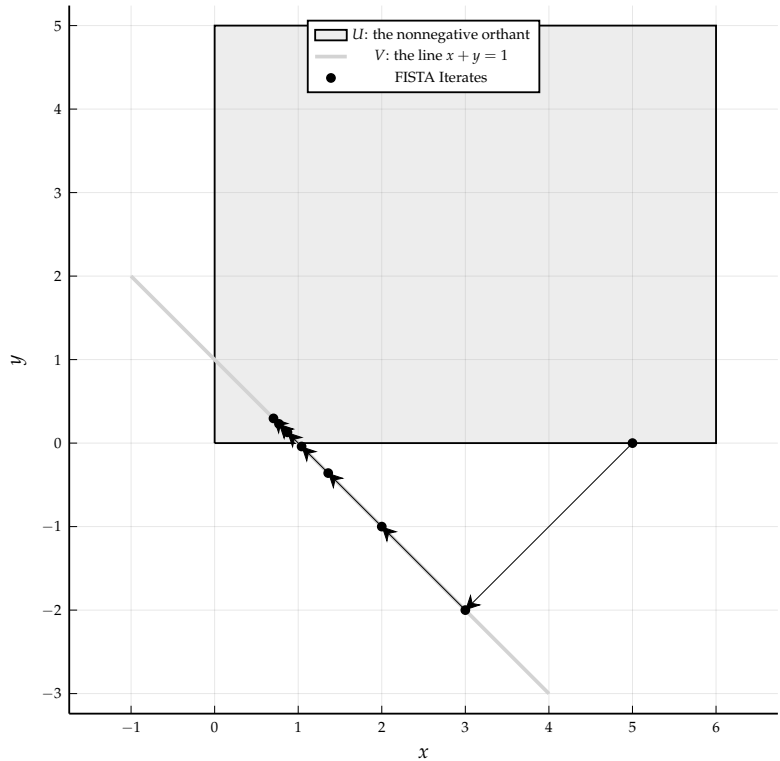
\begin{figure}[H] 
    \centering
    \scalebox{0.7}{

\begin{tikzpicture}[/tikz/background rectangle/.style={fill={rgb,1:red,1.0;green,1.0;blue,1.0}, fill opacity={1.0}, draw opacity={1.0}}, show background rectangle]
\begin{axis}[point meta max={nan}, point meta min={nan}, legend cell align={left}, legend columns={1}, title={}, title style={at={{(0.5,1)}}, anchor={south}, font={{\fontsize{14 pt}{18.2 pt}\selectfont}}, color={rgb,1:red,0.0;green,0.0;blue,0.0}, draw opacity={1.0}, rotate={0.0}, align={center}}, legend style={color={rgb,1:red,0.0;green,0.0;blue,0.0}, draw opacity={1.0}, line width={1}, solid, fill={rgb,1:red,1.0;green,1.0;blue,1.0}, fill opacity={1.0}, text opacity={1.0}, font={{\fontsize{8 pt}{10.4 pt}\selectfont}}, text={rgb,1:red,0.0;green,0.0;blue,0.0}, cells={anchor={center}}, at={(0.5, 0.98)}, anchor={north}}, axis background/.style={fill={rgb,1:red,1.0;green,1.0;blue,1.0}, opacity={1.0}}, anchor={north west}, xshift={1.0mm}, yshift={-1.0mm}, width={150.4mm}, height={150.4mm}, scaled x ticks={false}, xlabel={$x$}, x tick style={color={rgb,1:red,0.0;green,0.0;blue,0.0}, opacity={1.0}}, x tick label style={color={rgb,1:red,0.0;green,0.0;blue,0.0}, opacity={1.0}, rotate={0}}, xlabel style={at={(ticklabel cs:0.5)}, anchor=near ticklabel, at={{(ticklabel cs:0.5)}}, anchor={near ticklabel}, font={{\fontsize{11 pt}{14.3 pt}\selectfont}}, color={rgb,1:red,0.0;green,0.0;blue,0.0}, draw opacity={1.0}, rotate={0.0}}, xmajorgrids={true}, xmin={-1.7400000000000002}, xmax={6.74}, xticklabels={{$-1$,$0$,$1$,$2$,$3$,$4$,$5$,$6$}}, xtick={{-1.0,0.0,1.0,2.0,3.0,4.0,5.0,6.0}}, xtick align={inside}, xticklabel style={font={{\fontsize{8 pt}{10.4 pt}\selectfont}}, color={rgb,1:red,0.0;green,0.0;blue,0.0}, draw opacity={1.0}, rotate={0.0}}, x grid style={color={rgb,1:red,0.0;green,0.0;blue,0.0}, draw opacity={0.1}, line width={0.5}, solid}, axis x line*={left}, x axis line style={color={rgb,1:red,0.0;green,0.0;blue,0.0}, draw opacity={1.0}, line width={1}, solid}, scaled y ticks={false}, ylabel={$y$}, y tick style={color={rgb,1:red,0.0;green,0.0;blue,0.0}, opacity={1.0}}, y tick label style={color={rgb,1:red,0.0;green,0.0;blue,0.0}, opacity={1.0}, rotate={0}}, ylabel style={at={(ticklabel cs:0.5)}, anchor=near ticklabel, at={{(ticklabel cs:0.5)}}, anchor={near ticklabel}, font={{\fontsize{11 pt}{14.3 pt}\selectfont}}, color={rgb,1:red,0.0;green,0.0;blue,0.0}, draw opacity={1.0}, rotate={0.0}}, ymajorgrids={true}, ymin={-3.24}, ymax={5.24}, yticklabels={{$-3$,$-2$,$-1$,$0$,$1$,$2$,$3$,$4$,$5$}}, ytick={{-3.0,-2.0,-1.0,0.0,1.0,2.0,3.0,4.0,5.0}}, ytick align={inside}, yticklabel style={font={{\fontsize{8 pt}{10.4 pt}\selectfont}}, color={rgb,1:red,0.0;green,0.0;blue,0.0}, draw opacity={1.0}, rotate={0.0}}, y grid style={color={rgb,1:red,0.0;green,0.0;blue,0.0}, draw opacity={0.1}, line width={0.5}, solid}, axis y line*={left}, y axis line style={color={rgb,1:red,0.0;green,0.0;blue,0.0}, draw opacity={1.0}, line width={1}, solid}, colorbar={false}]
    \addplot[color={rgb,1:red,0.0;green,0.0;blue,0.0}, name path={117}, area legend, fill={rgb,1:red,0.6627;green,0.6627;blue,0.6627}, fill opacity={0.2}, draw opacity={1.0}, line width={1}, solid]
        table[row sep={\\}]
        {
            \\
            0.0  0.0  \\
            6.0  0.0  \\
            6.0  5.0  \\
            0.0  5.0  \\
            0.0  0.0  \\
        }
        ;
    \addlegendentry {$U$: the nonnegative orthant}
    \addplot[color={rgb,1:red,0.8275;green,0.8275;blue,0.8275}, name path={118}, draw opacity={1.0}, line width={2}, solid]
        table[row sep={\\}]
        {
            \\
            -1.0  2.0  \\
            4.0  -3.0  \\
        }
        ;
    \addlegendentry {$V$: the line $x + y = 1$}
    \addplot[color={rgb,1:red,0.0;green,0.0;blue,0.0}, name path={119}, only marks, draw opacity={1.0}, line width={0}, solid, mark={*}, mark size={2.25 pt}, mark repeat={1}, mark options={color={rgb,1:red,0.0;green,0.0;blue,0.0}, draw opacity={1.0}, fill={rgb,1:red,0.0;green,0.0;blue,0.0}, fill opacity={1.0}, line width={0.75}, rotate={0}, solid}]
        table[row sep={\\}]
        {
            \\
            5.0  0.0  \\
            3.0  -2.0  \\
            2.0  -1.0  \\
            1.3591232374373394  -0.35912323743733954  \\
            1.0404776519977057  -0.0404776519977057  \\
            0.8712565148187956  0.12874348518120446  \\
            0.7699305202142437  0.23006947978575626  \\
            0.7041777187727976  0.2958222812272025  \\
        }
        ;
    \addlegendentry {FISTA Iterates}
    \addplot[color={rgb,1:red,0.0;green,0.0;blue,0.0}, name path={120}, quiver={u={\thisrow{u}}, v={\thisrow{v}}, every arrow/.append style={-{Stealth[length = 8.0pt, width = 8.0pt]}}}, forget plot]
        table[row sep={\\}]
        {
            x  y  u  v  \\
            5.0  0.0  -2.0  -2.0  \\
            3.0  -2.0  -1.0  1.0  \\
            2.0  -1.0  -0.6408767625626606  0.6408767625626605  \\
            1.3591232374373394  -0.35912323743733954  -0.3186455854396337  0.31864558543963384  \\
            1.0404776519977057  -0.0404776519977057  -0.16922113717891007  0.16922113717891016  \\
            0.8712565148187956  0.12874348518120446  -0.10132599460455194  0.1013259946045518  \\
            0.7699305202142437  0.23006947978575626  -0.06575280144144613  0.06575280144144624  \\
        }
        ;
\end{axis}
\end{tikzpicture}}
    \caption{The first few iterates of FISTA
    with parameter sequence \cref{BTpar} and
    starting point $x_0=(5,0)$.
    Note that 
    $\lim_{k\to\infty} x_k \approx (0.4829,0.5171)$; in contrast, PGM converges to $(1,0)$.} 
    \label{fig:myplot}
\end{figure}
\end{example}

\section*{Acknowledgements}
{\color{black}
The authors thank an anonymous reviewer for insightful and helpful comments. 
}
The research of HHB and WMM was partially supported through 
Discovery Grants by the Natural Sciences and Engineering Research Council of Canada. 
The research  WMM was also partially supported through 
the Ontario Early Researcher Award.

\appendix

\renewcommand\thesection{\Alph{section}}

\numberwithin{equation}{section}
\numberwithin{proposition}{section}

\crefname{appendix}{appendix}{appendices}
\Crefname{appendix}{Appendix}{Appendices}
\crefalias{section}{appendix} 

\section{Proof of \cref{lem:key}}
\label{app:zkbd}

\begin{proof} 
We write $\xi_k$ instead of $\xi_k(s)$ for brevity. 
The left inequality 
in \cref{e:ineq1} follows 
because $\mu=\min F(X)$.
Let $k \ge 1$ and observe that 
\begin{equation}
 \|z_k - s\|^2 =
 \|t_k y_k +(1-t_k )x_k -s \|^2 .
\label{eq:z-norm-eq}
\end{equation}
It follows from \cref{eq:FS:ii},
\cref{eq:delta-def},
the convexity of $F$,
\cref{lem:29.2},
\cref{30.3a},
\cref{auxseq}, \cref{e:xyz} and \cref{eq:z-norm-eq}
that 
\begin{align*}
t_{k-1}^2 \delta_k - t_k^2 \delta_{k+1}
&\ge (t_k^2 - t_k)\delta_k - t_k^2\delta_{k+1} 
\\
&= t_k^2\!\left(\big(1-\tfrac{1}{t_k}\big)\!\big(F(x_k)-F(s)\big) - \big(F(x_{k+1})-F(s)\big)\right) 
\\
&= t_k^2\!\left(\big(1-\tfrac{1}{t_k}\big)F(x_k) + \tfrac{1}{t_k}F(s) - F(x_{k+1})\right) \\
&\ge t_k^2\!\left(F\!\left(\tfrac{1}{t_k}s + \big(1-\tfrac{1}{t_k}\big)x_k\right) - F(x_{k+1})\right)
\\
&\ge \frac{t_k^2 \lip}{2}\!\left(\Big\|\tfrac{1}{t_k}s + \big(1-\tfrac{1}{t_k}\big)x_k - x_{k+1}\Big\|^2
- \Big\|\tfrac{1}{t_k}s + \big(1-\tfrac{1}{t_k}\big)x_k - y_k\Big\|^2 \right) \\
&= \frac{\lip}{2}\!\big(\|s + (t_k-1)x_k - t_k x_{k+1}\|^2 - \|s + (t_k-1)x_k - t_k y_k\|^2\big) \\
&= \frac{\lip}{2}\!\big(\|z_{k+1} - s\|^2 - \|z_k - s\|^2\big)
.
\end{align*}
Rearranging yields
\begin{equation}
\xi_{k+1}=t_k^2 \delta_{k+1} + \frac{\lip}{2}\|z_{k+1} - s\|^2
\le
t_{k-1}^2 \delta_k + \frac{\lip}{2}\|z_k - s\|^2 =\xi_{k}.
\label{eq:master-ineq}
\end{equation}
This proves that
 \begin{equation}
\xi_{k+1}\le \xi_{k}\le \cdots \le \xi_1.
 \end{equation}
 We now turn to the right inequality in \cref{e:ineq1}. Observe that
 $\xi_1=t_0^2\big(F(x_1) - \mu\big) +\tfrac{\lip}{2} \|z_1 - s\|^2$.
Since $t_0=1$, $x_0=y_0$, and hence $y_1=x_1$ (by \cref{30.3b}), we have $z_1=x_1$ and thus
\begin{equation}
\xi_1=\big(F(x_1) - \mu\big) +\tfrac{\lip}{2} \|x_1 - s\|^2.
\end{equation}
Using \cref{lem:29.2} 
applied with $(x,y)$
replaced by $(s,x_0)$
and recalling $Tx_0 = Ty_0 = x_1$, we estimate
\begin{equation}
\mu - F(x_1) = F(s) - F(x_1)
\ge \tfrac{\lip}{2}\|x_1-s \|^2 - \tfrac{\lip}{2}\|x_0-s \|^2;
\end{equation}
equivalently,
\begin{equation}
F(x_1) - \mu
\le \frac{\lip}{2} \big(\|x_0-s \|^2 - \|x_1-s \|^2 \big).
\label{eq:309}
\end{equation}
Altogether,
we learn that 
$\xi_1\le \tfrac{\lip}{2}\norm{x_0-s}^2$
and \cref{e:ineq1} is verified.

It now follows from \cref{e:ineq1}
that $(\xi_k)_{k\in\NN}$
is a bounded decreasing sequence of real numbers, hence 
convergent.

Finally, because 
$(\xi_k)_{k\in\NN}$
is the sum of two nonnegative sequences, 
namely $(t_{k-1}^2 \delta_k )_{k\in\NN}$
 and $ (\frac{\lip}{2}\|z_k - s\|^2)_{k\in\NN}$, 
the boundedness of $(z_k)_\knn$ follows. 
 \end{proof}

\section{Proof of \cref{l:Radu} when $\ell=\pm\infty$}

\label{app:ellpm}


\begin{proof}
Assume first that $\ell=+\infty$; i.e., 
$g_k\to+\infty$. 
Fix an arbitrary ``hurdle'' $M>0$, and 
get $K\in\NN$ such that 
$(\forall k\geq K)$ 
$g_k\geq M$. 
Using \cref{260112e}, we estimate, 
for all $n\geq K$, 
\begin{subequations}
\begin{align}
\label{260112m}
h_n  - h_0\prod_{j=0}^{n-1}\lambda_j 
&= 
\sum_{k=0}^{n-1} w_{n,k}g_k
= 
\sum_{k=0}^{K-1} w_{n,k}g_k
+
\sum_{k=K}^{n-1} w_{n,k}g_k
\geq 
\sum_{k=0}^{K-1} w_{n,k}g_k
+
M\sum_{k=K}^{n-1} w_{n,k}\\
&= 
\sum_{k=0}^{K-1} w_{n,k}(g_k-M)
+ M\sum_{k=0}^{n-1} w_{n,k}. 
\label{260112l}
\end{align}
\end{subequations}
First, the product in \cref{260112m} converges 
to $0$ by \cref{260112c}. 
Now consider the two sums in \cref{260112l}:
The left sum converges to $0$ as $n\to\infty$ 
by \cref{260112h} while 
the right sum converges to $M$ by 
\cref{260112f}. 
Altogether, we deduce that 
\begin{equation}
\varliminf_{n\to\infty} h_n \geq M. 
\end{equation}
Now letting $M\to +\infty$ yields
$\varliminf_{n\to\infty}h_n= +\infty$ and the 
conclusion follows in this case. 

The proof for $\ell=-\infty$ is similar.
(Alternatively, 
apply the just-proven $\ell=+\infty$ result to the sequence 
$(-h_k)_\knn$.) 
 \end{proof}

\section{Proof of \cref{ex:sinh}}

\label{app:sinh}

\begin{proof}
For $k\geq 1$, we have 
\begin{align}
g_k 
&= 
h_{k+1}+\varphi_k(h_{k+1}-h_k)
= \frac{\varphi_k}{1+\varphi_k}h_k 
+ \varphi_k \Big(
\frac{\varphi_k}{1+\varphi_k}h_k -h_k\Big)\\
&= 
\frac{\varphi_k}{1+\varphi_k}
h_k\big(1+\varphi_k - (1+\varphi_k) \big) = 0 = \ell. 
\end{align}
Now (recall $h_1 = 1$)
\begin{equation}
h_k = h_1\prod_{j=1}^{k-1}\frac{j^2}{1+j^2}
\to \prod_{j=1}^\infty\frac{j^2}{1+j^2}. 
\end{equation}
On the other hand, we have the classical Euler formula 
(see, e.g., \cite[Section~2.2]{Melnikov})
\begin{equation}
\frac{\sinh(\pi x)}{\pi x}
= \prod_{n=1}^\infty \Big(1+\frac{x^2}{n^2} \Big)
\end{equation}
which yields (upon setting $x=1$) 
\begin{equation}
\frac{\sinh(\pi)}{\pi}
= \prod_{n=1}^\infty \Big(1+\frac{1}{n^2} \Big) 
= \prod_{n=1}^\infty \frac{n^2+1}{n^2}. 
\end{equation}
Altogether, 
\begin{equation}
h_k \to \prod_{j=1}^\infty\frac{j^2}{1+j^2} 
= \Big(\prod_{j=1}^\infty\frac{1+j^2}{j^2}\Big)^{-1}
= \frac{\pi}{\sinh(\pi)}
\approx 0.272 > 0 = \ell. 
\end{equation}
\end{proof}
                                        
\section{Proof of \cref{p:bonusSalzo}}

\label{app:Salzo}

\begin{proof}
The linearity of the inner product in each variable shows that 
$(\scal{x_k}{y})_\knn$ converges for 
all $y\in\spn C$. 
Next, the continuity of the inner product shows that 
$(\scal{x_k}{y})_\knn$ converges for 
all $y\in Y= \cspan C$. 
Thus 
\begin{equation}
\label{260113d}
(\scal{x_k}{y})_\knn 
\;\;\text{converges for all $y\in Y$, 
a closed linear subspace of $X$.}
\end{equation}
Thinking of $y\in Y$ as $y=P_Yx$, 
\cref{260113d} is equivalent to 
\begin{equation}
(\scal{P_Yx_k}{x})_\knn 
\;\;\text{converges for all $x\in X$.} 
\end{equation}
By the Principle of Uniform Boundedness (Banach-Steinhaus), 
\begin{equation}
(y_k)_\knn := (P_Yx_k)_\knn 
\;\;\text{is bounded.}
\end{equation}
Note that $\ell(x) := \lim_{n\to\infty} \scal{x}{y_n}$ exists for all $x\in X$. 
The function $\ell$ is clearly linear. 
Moreover, $\ell(x) \leq \|x\|\sup_\knn \|y_k\|$ and so 
$\|\ell\|\leq \sup_\knn\|y_k\|<+\infty$.
So $\ell$ is also continuous. 
By the Riesz Representation Theorem,
there exists 
$\overline{y}\in X$ such that 
$(\forall x\in X)$
$\lim_{k\to\infty} \scal{x}{P_Yx_k} = \ell(x) = \scal{x}{\overline{y}}$, i.e., $(P_Yx_k)_\knn$ converges weakly  to $\overline{y}$. 
\end{proof}

\end{document}